\documentclass[12pt]{amsart}
\usepackage{a4wide,enumerate,xcolor,graphicx}
\usepackage{amsmath}
\usepackage{scrextend} 
\allowdisplaybreaks

\let\pa\partial
\let\na\nabla
\let\eps\varepsilon
\newcommand{\N}{{\mathbb N}}
\newcommand{\R}{{\mathbb R}}

\newcommand{\diver}{\operatorname{div}}

\newcommand{\dom}{{\mathcal O}}
\newcommand{\dd}{{\mathrm d}}
\newcommand{\E}{\mathbb{E}}
\newcommand{\F}{{\mathcal F}}
\newcommand{\Prob}{\mathbb{P}}

\newtheorem{theorem}{Theorem}
\newtheorem{lemma}[theorem]{Lemma}

\newtheorem{corollary}[theorem]{Corollary}
\newtheorem{definition}{Definition}

\begin{document}

\title[Martingale solutions to a cross-diffusion system]{Global martingale solutions 
to a segregation cross-diffusion system with stochastic forcing} 

\author[M. Biswas]{Mrinmay Biswas}
\address{Indian Institute of Technology Kanpur, Department of Mathematics and 
Statistics, Kanpur 208016, India}
\email{mbiswas@iitk.ac.in} 

\author[A. J\"ungel]{Ansgar J\"ungel}
\address{Institute of Analysis and Scientific Computing, Vienna University of  
	Technology, Wiedner Hauptstra\ss e 8--10, 1040 Wien, Austria}
\email{juengel@tuwien.ac.at} 

\date{\today}

\thanks{The authors acknowledge partial support from   
the Austrian Science Fund (FWF), grants P33010, W1245, and F65.
This work has received funding from the European 
Research Council (ERC) under the European Union's Horizon 2020 research and 
innovation programme, ERC Advanced Grant no.~101018153.} 

\begin{abstract}
The existence of a global martingale solution to a cross-diffusion system with
multiplicative Wiener noise in a bounded domain with no-flux boundary conditions
is shown. The model describes
the dynamics of population densities of different species due to segregation
cross-diffusion effects. The diffusion matrix is generally neither symmetric 
nor positive semidefinite. This difficulty is overcome by exploiting
the Rao entropy structure. The existence proof uses a stochastic Galerkin method,
uniform estimates from the Rao entropy inequality, 
and the Skorokhod--Jakubowski theorem.
Furthermore, an exponential equilibration result is proved for sufficiently
small Lipschitz constants of the noise by using the
relative Rao entropy. Numerical tests illustrate the behavior of solutions in
one space dimension for two and three population species.
\end{abstract}

 

\maketitle


\section{Introduction}

The segregation of population species can be described by cross-diffusion systems
involving quadratic nonlinearities; see, e.g., \cite{BGHP85}. This class of models
was derived as the mean-field limit of moderately interacting particle systems
\cite{CDJ19}. In this paper, we analyze these systems taking into account the random
influence of the environment. We assume that the dynamics of the population density
$u_i$ of the $i$-th species is modeled by
\begin{equation}\label{1.eq}
  \dd u_i(t) = \diver\big(\delta\na u_i + u_i\na p_i(u)\big)\dd t
	+ \sum_{j=1}^n \sigma_{ij}(u)\dd W_j(t), \quad p_i(u) = \sum_{j=1}^n a_{ij}u_j,
\end{equation}
in a bounded domain $\dom\subset\R^d$ ($d\ge 1$)
with the initial and no-flux boundary conditions
\begin{equation}\label{1.bic}
  u_i(0)=u_i^0\quad\mbox{in }\dom, \quad
	\na u_i\cdot\nu=0\quad\mbox{on }\pa\dom,\ t>0,\ i=1,\ldots,n,
\end{equation}
where $\nu$ is the exterior unit normal vector to $\pa\dom$, 
$W_1,\ldots,W_n$ are cylindrical Wiener processes in a Hilbert space $U$
(see Section \ref{sec.main} for details),
and $\delta>0$, $a_{ij}\ge 0$ for 
$i,j=1,\ldots,n$. The density $u_i$ depends on the spatial variable
$x\in\dom$, the time $t\ge 0$, and the stochastic variable $\omega\in\Omega$. 
The terms $a_{ij}u_i \na u_j$ for $i\neq j$ are called cross-diffusion terms.


The analysis of equations \eqref{1.eq} is delicate 
already in the deterministic case, where
$\sigma_{ij}(u)=0$, since the diffusion matrix is generally neither symmetric
nor positive semi-definite. The key idea of the analysis is to exploit the
entropy structure associated to \eqref{1.eq}.

To explain this structure, we need two assumptions.
First, we suppose that the
diffusion matrix $(u_i a_{ij})\in\R^{n\times n}$ 
has only eigenvalues with positive real part (for $u_i>0$). 
This means that system \eqref{1.eq} is parabolic in the
sense of Petrovskii, which is a minimal condition for local solvability
\cite{Ama93}. Second, we assume the existence of numbers $\pi_i>0$
satisfying 
\begin{equation*}
  \pi_i a_{ij} = \pi_j a_{ji} \quad\mbox{for all }i,j=1,\ldots,n,
\end{equation*}
which is the detailed-balance condition for the Markov chain associated to 
$(a_{ij})$, and $(\pi_1,\ldots,\pi_n)$ is the corresponding reversible 
stationary measure. Both assumptions imply that the matrix $(\pi_ia_{ij})$ 
is symmetric positive definite.
We introduce the so-called Rao entropy 
\begin{equation}\label{1.H}
  H(u) = \int_\dom h(u)\dd x, \quad
	h(u) = \frac12\sum_{i,j=1}^n\pi_i a_{ij}u_iu_j,
\end{equation}
which is used as a diversity measure in population dynamics \cite{Rao82}.
Observe that the entropy density $h(u)$ is not the sum of the individual entropies
of the species, but it mixes the species.
A formal computation, which is made
rigorous below (see Lemma \ref{lem.eiN}), shows that
\begin{equation}\label{1.ei}
  \frac{\dd}{\dd t}\E H(u(t)) + \delta\sum_{i,j=1}^n\pi_i a_{ij}\E\int_\dom
	\na u_i\cdot\na u_j\dd x + \sum_{i=1}^n\pi_i\E\int_\dom u_i|\na p_i(u)|^2 \dd x = 0.
\end{equation}
Since $(\pi_ia_{ij})$ is positive definite and $\delta>0$, this yields
an a priori estimate for $\na u_i$ in $L^2(\dom)$. The last integral can be
interpreted as the kinetic energy of the system, with $\na p_i(u)$ being the
partial velocity. Unfortunately, it does not provide any gradient estimate, and
this is the reason why we have included the $\delta$-terms. 
The main results of this paper are the existence of a global martingale solution
to \eqref{1.eq}--\eqref{1.bic} and the exponential decay in expectation of the
solution to its spatial average. These results are detailed in Section \ref{sec.main}.

The analysis of stochastic cross-diffusion systems is rather recent. 
The existence of martingale solutions to cross-diffusion systems with a
positive definite diffusion matrix (including nonlocal diffusion) was shown in 
\cite{BeKa22}. A stochastic population cross-diffusion system with cubic
nonlinearities was investigated in \cite{DJZ19}. 
In these works, the quadratic
energy structure allows for the use of a stochastic Galerkin method. 
Combining the theory of quasilinear parabolic equations with evolution semigroup
methods, the authors of \cite{KuNe20} 
proved the existence of a unique local pathwise mild
solution to the Shigesada--Kawasaki--Teramoto system with stochastic forcing. 

For the existence analysis of system \eqref{1.eq}, we use similar techniques
as in \cite{DJZ19}. The large-time behavior result is based on the relative entropy
method (which becomes here a weighted $L^2$ norm). Both results are new.

The random influence may also be taken into account on the level of the fluxes,
using a conservative noise of the type $\sum_{j=1}^n\diver\sigma_{ij}(u)\dd W_j$.
This noise can be handled by considering the noise in $H^{-1}(\dom)$ and estimating
in the norm of that space \cite{Cio20}. Writing the noise in Stratonovich form
and assuming that $\diver(\sigma(u)/u)=0$,
the conservative noise can be reformulated as an It\^o noise plus a regularizing
Laplacian term; see \cite{MaTo21}. The authors of \cite{FeGe19} pass to the
kinetic formulation of the equation, which yields
an equation in which the noise enters as a linear transport.
Unfortunately, these techniques cannot be easily applied to equations
\eqref{1.eq} with arbitrary conservative noise.



\section{Notation and main results}\label{sec.main}

Let $\dom\subset\R^d$ ($d\ge 1$) be a bounded domain and
let $(\Omega,\F,\mathbb{F},\Prob)$ be a filtered probability space endowed with a
complete right-continuous filtration $\mathbb{F}=(\F_t)_{t\ge 0}$.
We write $L^p(\Omega,\F;B)$ or simpler $L^p(\Omega;B)$ 
for the set of all $\F$-measurable $B$-valued
random variables in a Banach space $B$, such that 
$\E\|u\|_B^p=\int_\Omega\|u(\omega)\|_B^p\Prob(\dd\omega)
<\infty$ ($1\le p<\infty$). 
Let $U$ be a Hilbert space with orthonormal basis
$(\eta_k)_{k\in\N}$. The space of Hilbert--Schmidt operators from $U$ to $L^2(\dom)$ 
is defined by
$$
  \mathcal{L}_2(U;L^2(\dom)) = \bigg\{F:U\to L^2(\dom)\mbox{ linear, continuous}:
	\sum_{k=1}^\infty\|F\eta_k\|_{L^2(\dom)}<\infty\bigg\},
$$
endowed with the norm $\|F\|_{\mathcal{L}_2(U;L^2(\dom))}
=(\sum_{k=1}^\infty\|F\eta_k\|_{L^2(\dom)}^2)^{1/2}$. Let $\beta_{jk}$ for
$j=1,\ldots,n$, $k\in\N$ be independent one-dimensional Brownian motions.
There exists a Hilbert space $U_0\supset U$
and a Hilbert--Schmidt embedding $J:U\to U_0$ such that the series
$W_j=\sum_{k=1}^\infty\beta_{jk}J(\eta_k)$ converges in 
$L^\infty(0,T;L^2(\Omega; U_0))$. Moreover, 
$W_j(\omega)\in C^0([0,T];U_0)$ for a.e.\ $\omega\in\Omega$
\cite[Prop.~2.5.2]{LiRo15}.

We impose the following hypotheses:

\begin{itemize}
\item[(H1)] Domain: Let $\dom\subset\R^d$ ($1\le d\le 3$) be a bounded domain with
Lipschitz boundary.
\item[(H2)] Initial datum: $u^0_i\in L^2(\Omega,\F_0;L^2(\dom))$ satisfies
$u_i^0\ge 0$ a.e.\ in $\dom$, $\Prob$-a.s.\ for $i=1,\ldots,n$.
\item[(H3)] Diffusion coefficients: All eigenvalues of $A=(a_{ij})\in\R^{n\times n}$ 
have positive real parts and the detailed-balance
condition holds, i.e., there exist $\pi_1,\ldots,\pi_n>$ such that
$$
  \pi_i a_{ij} = \pi_j a_{ji} \quad\mbox{for all }i,j=1,\ldots,n.
$$
\item[(H4)] Noise coefficients: $\sigma_{ij}:L^2(\dom)\times\Omega\to 
\mathcal{L}_2(U;L^2(\dom))$ is $\mathcal{B}(L^2(\dom)\otimes\F/
\mathcal{B}(\mathcal{L}_2(U;L^2(\dom)))$-measurable and $\mathbb{F}$-adapted,
and there exists $C_\sigma>0$ such that for all $i,j=1,\ldots,n$ and
$u,v\in L^2(\dom)$,
\begin{align*}
  \|\sigma_{ij}(u)-\sigma_{ij}(v)\|_{\mathcal{L}_2(U;L^2(\dom))}
	&\le C_\sigma\|u-v\|_{L^2(\dom)}, \\
	\|\sigma_{ij}(u)\|_{\mathcal{L}_2(U;L^2(\dom))} 
	&\le C_\sigma\|u_i\|_{L^2(\dom)}.
\end{align*}
\end{itemize}

Let us comment on these hypotheses. 
The restriction $d\le 3$ is not essential; our proof works in any space dimension,
but we need to choose a larger space to show the Aldous condition, needed for the 
tightness of the laws of the Galerkin solutions.
Let $P=\operatorname{diag}(\pi_1,\ldots,\pi_n)$.
Then, by Hypothesis (H3), $PA$ is symmetric. 
Since the eigenvalues of $PA$ are positive, we conclude that $PA$ is
positive definite.
The linear growth condition in Hypothesis (H4) is needed to prove the
$\Prob$-a.s.\ nonnegativity of $u_i$. It can be weakened to
$\|\sigma_{ij}(u)\|_{\mathcal{L}_2(U;L^2(\dom))}\le C_\sigma(1+\|u\|_{L^2(\dom)})$,
but then we cannot prove the nonnegativity property.

\begin{definition}[Martingale solution]
Let $T>0$. A {\em global martingale solution} to \eqref{1.eq}--\eqref{1.bic} is a tuple
$(\widetilde V,\widetilde W,\widetilde u)$ such that
\begin{itemize}
\item $\widetilde V = (\widetilde\Omega,\widetilde{\mathcal{F}},\widetilde\Prob,
\widetilde{\mathbb{F}})$
is a complete probability space with a filtration
$\widetilde{\mathbb{F}}=(\widetilde{\mathcal{F}}_t)_{t\in[0,T]}$;
\item $\widetilde W=(\widetilde W_1,\ldots,\widetilde W_n)$ 
is a cylindrical Wiener process in $U^n:=U\times\cdots\times U$;
\item $\widetilde u=(\widetilde u_1,\ldots,\widetilde u_n):[0,T]\times\Omega
\to L^2(\dom)$ is an $\mathbb{F}$-progressively measurable process such that
for all $t\in[0,T]$ and $i=1,\ldots,n$,
$$
  \widetilde u_i\in L^2(\Omega;C^0([0,T];L_w^2(\dom)))\cap 
	L^2(\Omega;L^2(0,T;H^1(\dom))),
$$
the law of $\widetilde u_i(0)$ is the same as for $u_i^0$, and $\widetilde u_i$
satisfies for $t\in [0,T]$ and $\phi\in H^1(\dom)$,
\begin{align*}
  (\widetilde u_i(t),\phi)_{L^2(\dom)} &= (\widetilde u_i(0),\phi)_{L^2(\dom)}
	- \int_0^t\int_\dom\big(\delta\na\widetilde u_i(s) + \widetilde u_i(s)
	\na p_i(\widetilde u(s))\big)\cdot\na\phi \dd x\dd s \\
	&\phantom{xx}{}-\int_0^t\bigg(\sum_{j=1}^n\sigma_{ij}(\widetilde u(s))
	\dd \widetilde W_j(s),\phi\bigg)_{L^2(\dom)}\quad\Prob\mbox{-a.s.}
\end{align*}
\end{itemize}
\end{definition}

Our first main result is as follows.

\begin{theorem}[Existence of a global martingale solution]\label{thm.ex}
Let $T>0$ and let Hypotheses (H1)--(H4) hold. Then there exists a global
martingale solution $(\widetilde U,\widetilde W,\widetilde u)$ to 
\eqref{1.eq}--\eqref{1.bic} satisfying $\widetilde u_i(t)\ge 0$ a.e.\ in $\dom$,
$\widetilde \Prob$-a.s.\ for all $t\in[0,T]$, $i=1,\ldots,n$.
\end{theorem}

As mentioned in the introduction, the idea of the proof is the use of the stochastic
Galerkin method combined with the entropy method. We project
equations \eqref{1.eq} on a Galerkin space with finite dimension $N\in\N$. The existence
of a pathwise unique strong solution $u^N=(u_1^N,\ldots,u_n^N)$ 
(up to some stopping time) is shown by means of 
Banach's fixed-point theorem. It\^{o}'s lemma allows us to derive an entropy
inequality; see \eqref{1.ei}. 
Since we introduced the $\delta$-terms, we obtain estimates for $u^N$ in
$H^1(\dom)$ uniformly in $N$. 

The tightness of the laws of $(u^N)$ in the topological
space $Z_T$, defined in \eqref{3.ZT}, is proved by applying the criterion of
\cite{ZGJ13}. We deduce from Skorokhod--Jakubowski's theorem that there exists
a subsequence of $(u^N)$,
another probability space, and random variables $(\widetilde{u}^N,\widetilde{W}^N)$
having the same law as $(u^N,W^N)$, and $(\widetilde{u}^N,\widetilde{W}^N)$
converges to $(\widetilde{u},\widetilde{W})$ in the topology of $Z_T$.
Because of the gradient estimates and compactness, we infer the strong convergence
$\widetilde{u}^N\to\widetilde{u}$ in $L^2(\dom)$ a.s., and we can identify
the limits in the nonlinearities. Finally, the a.s.\ nonnegativity follows
from a stochastic Stampacchia truncation argument from \cite{CPT16}.

For the second main result, we introduce the relative Rao entropy
$$
  H(u|\bar{u}) 
	= H(u) - H(\bar{u}) - \frac{\delta H}{\delta u}(\bar{u})\cdot(u-\bar{u})
	= \frac12\sum_{i,j=1}^n\pi_i a_{ij}\int_\dom(u_i-\bar{u}_i)(u_j-\bar{u}_j)\dd x,
$$
where $\bar{u}$ is the solution to
\begin{equation*}
  \dd\bar{u}_i(t) = \sum_{j=1}^n\bar\sigma_{ij}(u(t))\dd W_j(t), \quad t>0, \quad
	\bar\sigma_{ij}(u) = \frac{1}{|\dom|}\int_\dom
	\sigma_{ij}(u)\dd x,
\end{equation*}
and $|\dom|$ is the measure of $\dom$.
Recalling $P=\operatorname{diag}(\pi_1,\ldots,\pi_n)$ and $A=(a_{ij})\in\R^{n\times n}$,
the relative entropy can be written as the weighted $L^2$ norm
$$
  H(u|\bar{u}) = \frac12\|(PA)^{1/2}(u-\bar{u})\|_{L^2(\dom)}^2,
$$

\begin{theorem}[Exponential time decay]\label{thm.time}
Let Hypotheses (H1)--(H4) hold and let $u$ be a martingale solution to
\eqref{1.eq}--\eqref{1.bic}. Then there exists $c_0>0$ such that for all
$0<C_\sigma<c_0$ (see Hypothesis (H4)),
$$
  \E H(u(t)|\bar{u}(t)) \le \E H(u^0|\bar{u}(0))e^{-\eta t}, \quad
	\eta := c_0^2-C_\sigma^2,\ t>0.
$$
\end{theorem}

Usually, the large-time behavior of solutions to stochastic differential equations
is analyzed by proving the existence and uniqueness of an invariant measure
and studying the ergodicity of the equations. This can be done, for instance, 
by establishing the strong Feller property and the tightness of the laws of the
solutions and by applying the Krylov--Bogolyubov's theorem \cite[Chapter 11]{DaZa14}.
Here, this program is delicate since we are lacking semigroup properties.
Therefore, we rely on estimates from the entropy method.


For the proof of Theorem \ref{thm.time}, we apply It\^{o}'s lemma to the
process $e^{-\eta t/2}(PA)^{1/2}(u-\bar{u})$, take the expectation, and
use the Poincar\'e--Wirtinger inequality with constant $C_P>0$
for the gradient term. This leads to (see Section \ref{sec.time})
\begin{align}
  e^{\eta t}&\E H(u(t)|\bar{u}(t)) - \E H(u(0)|\bar{u}(0)) - \eta\int_0^t e^{\eta s}
	\E H(u(s)|\bar{u}(s))\dd s \nonumber \\
	&\le -\frac{\delta\lambda}{C_P^2}\int_0^t\|u-\bar{u}\|_{L^2(\dom)}^2\dd s
	+ \frac12\E\int_0^t\|(PA)^{1/2}(\sigma(u)-\bar{\sigma}(u))
	\|_{\mathcal{L}_2(U;L^2(\dom))}^2\dd s. \label{1.aux}
\end{align}
The difficult part of the proof is the estimate of the last integral.
Using the properties of $\sigma_{ij}(u(t))$ and $\bar{\sigma}_{ij}(u(t))$,
we find that
$$
  \E\int_0^t\|(PA)^{1/2}(\sigma(u)-\bar{\sigma}(u))
	\|_{\mathcal{L}_2(U;L^2(\dom))}^2\dd s 
	\le 2C_1C_\sigma^2\|u-\bar{u}\|_{L^2(\dom)}^2.
$$
Thus choosing $C_\sigma>0$ (see Hypothesis (H4)) and $\eta>0$
sufficiently small, this integral can be absorbed by the first term on the
right-hand side of \eqref{1.aux}, and we conclude the result.


\section{Proof of Theorem \ref{thm.ex}}

The proof of Theorem \ref{thm.ex} is split into several steps.

\subsection{Stochastic Galerkin approximation}

Let $(e_k)_{k\in\N}$ be an orthonormal basis of $L^2(\dom)$, which is 
orthogonal in $H^1(\dom)$. 
For each $N\in\N$,
define the finite-dimensional subspace $H_N:=\operatorname{span}(e_1,\ldots,e_N)$
of $L^2(\dom)$ and the corresponding projection $\Pi_N:L^2(\dom)\to H_N$,
$\Pi_N(v) = \sum_{i=1}^N(v,e_i)_{L^2(\dom)}e_i$ for $v\in L^2(\dom)$.
We consider system \eqref{1.eq} projected on the subspace 
$H_N^n:=H_n\times\cdots\times H_N$:
\begin{align}\label{2.eqN}
  \dd u_i^N(t) &= \Pi_N\diver\big(\delta\na u_i^N 
	+ (u_i^N)^+\na p_i(u^N)\big)\dd t
	+ \Pi_N\sum_{j=1}^n\sigma_{ij}(u^N)\dd W_j(t), \quad t>0, \\
	u_i^N(0) &= \Pi_N(u_i^0), \quad i=1,\ldots,n, \label{2.icN}
\end{align}
where $z^+=\max\{0,z\}$ denotes the positive part. This truncation is
necessary, since the Rao entropy does not provide nonnegative densities. The
nonnegativity of the Galerkin limit $u_i$ is proved in Section \ref{sec.end}.
Given $T>0$, we introduce the space $M_T:=L^2(\Omega;C^0([0,T];H_N^n))$ with the
norm $\|u\|_{M_T}=(\E\sup_{0<t<T}\|u(t)\|_{L^2(\dom)}^2)^{1/2}$ for $u\in M_T$.
For given $R>0$ and $u\in M_T$, we define the exit time $\tau_R=\inf\{t\in[0,T]:
\|u(t)\|_{H^1(\dom)}>R\}$. Then $\{\omega\in\Omega:\tau_R(\omega)>t\}$ belongs
to $\F_t$ for every $t\in[0,T]$, and $\tau_R$ is an $\mathbb{F}$-stopping time.

\begin{lemma}[Local existence of $u^N$]\label{lem.ex.uN}
Let $T>0$, $R>0$ and let the Hypotheses (H1)--(H4) hold. 
Then there exists a pathwise unique strong solution $u^N\in M_{T\wedge\tau_R}$ to 
\eqref{2.eqN}--\eqref{2.icN} such that for any $t\in[0,T\wedge\tau_R]$,
$i=1,\ldots,n$, and $\phi=(\phi_1,\ldots,\phi_n)\in H_N^n$,
\begin{align*}
  (u_i^N(t),\phi_i)_{L^2(\dom)} &= (u_i(0),\phi_i)_{L^2(\dom)}
	- \int_0^t\int_\dom\big(\delta\na u_i^N(s) + (u_i^N)^+(s)\na p_i(u^N(s))\big)
	\cdot\na\phi_i\dd x\dd s \\
	&\phantom{xx}{}
	+ \bigg(\sum_{j=1}^n\int_0^t\sigma_{ij}(u^N(s))\dd W_j(s),\phi_i\bigg)_{L^2(\dom)}
	\quad\Prob\mbox{-a.s.}
\end{align*}
\end{lemma} 

\begin{proof}
We use Banach's fixed-point theorem. Define the fixed-point operator 
$S:M_T\to M_T$ for $v\in M_T$ and $\psi\in H_N^n$ by
\begin{align*}
  (S(v)(t),\psi)_{L^2(\dom)} &= \sum_{i=1}^n(\Pi_N u_i^0,\psi_i)_{L^2(\dom)}
	- \sum_{i=1}^n\int_0^t\int_\dom\big(\delta\na v_i 
	+ v_i^+\na p_i(v)\big)\cdot\na\psi_i\dd x\dd s \\
	&\phantom{xx}{}
	+ \sum_{i=1}^n\bigg(\sum_{j=1}^n\int_0^t\sigma_{ij}(v(s))\dd W_j(s),\psi_i
	\bigg)_{L^2(\dom)}\quad\Prob\mbox{-a.s.}, \ t\in[0,T].
\end{align*}
The aim is to show that $S$ is a contraction on $M_{T\wedge\tau_R}$ for sufficiently
small $T>0$. 

First, we verify that $S$ is a self-mapping.
Let $v\in M_T$, $\psi\in H_N^n$, and $T_R:=T\wedge\tau_R$. Then
\begin{align}\label{2.aux}
  & \|(S(v),\psi)_{L^2(\dom)}\|_{L^2(\Omega;L^\infty(0,T_R))}^2
	\le I_1+\cdots+I_4, \quad\mbox{where} \\
	& I_1 = \E\|u^0\|_{L^2(\dom)}^2\|\psi\|_{L^2(\dom)}^2, \nonumber \\
	& I_2 = \E\bigg(\sup_{0<t<T_R}\bigg|\sum_{i=1}^n\int_0^t\int_\dom
	\delta\na v_i(s)\cdot\na\psi_i\dd x\dd s\bigg|\bigg)^2, \nonumber \\
	& I_3 = \E\bigg(\sup_{0<t<T_R}\bigg|\sum_{i,j=1}^n\int_0^t\int_\dom
	a_{ij} v_i^+(s)\na v_j(s)\cdot\na \psi_i\dd x\dd s\bigg|\bigg)^2, \nonumber \\
	& I_4 = \E\bigg(\sup_{0<t<T_R}\bigg|\sum_{i=1}^n\bigg(\sum_{j=1}^n\int_0^t
	\sigma_{ij}(v)\dd W_j(s),\psi_i\bigg)_{L^2(\dom)}\bigg|\bigg)^2. \nonumber 
\end{align}
We estimate the term $I_2$, using the equivalence of the norms in $H_N$:
\begin{align*}
  I_2 &\le \delta^2 T\E\bigg(\sup_{0<t<T_R}\sum_{i=1}^n\int_0^t\bigg|
	\int_\dom\na v_i(s)\cdot\na\psi_i\dd x\bigg|^2\dd s\bigg) \\
	&\le \delta^2 T C\E\bigg(\int_0^{T_R}\|\na v(s)\|_{L^2(\dom)}^2\dd s\bigg)
	\|\na\psi\|_{L^2(\dom)}^2 \\
	&\le C(N)\delta^2 T^2\E\bigg(\sup_{0<s<T_R}\|v(s)\|_{L^2(\dom)}^2\bigg)
	\|\psi\|_{L^2(\dom)}^2
	= C(N)\delta^2 T^2\E\|v\|_{M_{T_R}}^2\|\psi\|_{L^2(\dom)}^2,
\end{align*}
where here and in the following $C>0$, $C_i>0$, etc.\ are constants independent
of the solution, with values changing from line to line.
The term $I_3$ is estimated in a similar way, taking into account that
$\|\na v(s)\|_{L^2(\dom)}\le R$ for $s<T_R$:
\begin{align*}
	I_3 &\le C T\E\bigg(\int_0^{T_R}\|v(s)\|_{L^2(\dom)}^2
	\|\na v(s)\|_{L^2(\dom)}^2\dd s\bigg)\|\na\psi\|_{L^\infty(\dom)}^2 \\
	&\le C(N)R^2T^2\E\bigg(\sup_{0<s<T_R}\|v(s)\|_{L^2(\dom)}^2\bigg)
	\|\psi\|_{L^2(\dom)}^2,
\end{align*}
Finally, by the 
Burkholder--Davis--Gundy inequality \cite[Theorem 1.1.7]{LoRo17} and (H4),
\begin{align*}
  I_4 &\le C\E\bigg(\int_0^{T_R}\|\sigma(v(s))\|_{\mathcal{L}_2(U;L^2(\dom))}^2\dd s
	\bigg)\|\psi\|_{L^2(\dom)}^2 \\
	&\le C_\sigma^2C\E\bigg(\int_0^{T_R}\|v(s)\|_{L^2(\dom)}^2\dd s\bigg)
	\|\psi\|_{L^2(\dom)}^2 \le C_\sigma^2 C(N)T\|v\|_{M_{T_R}}^2\|\psi\|_{L^2(\dom)}^2.
\end{align*}
Inserting these estimates into \eqref{2.aux} leads to
$$
  \|S(v)\|_{M_{T_R}}^2 \le \E\|u^0\|_{L^2(\dom)}^2
	+ C(N,R)(T^2+T)\|v\|_{M_{T_R}}^2.
$$
This shows that $S$ is a self-mapping. Next, we prove that $S$ is a contraction
if $T>0$ is sufficiently small. The estimations are similar as above with the
exception of the nonlinear diffusion part. Let $u,v\in M_T$ and $\psi\in H_N^n$.
Then
\begin{align}\label{2.aux2}
  & \|(S(u)-S(v),\psi)_{L^2(\dom)}\|_{L^2(\Omega;L^\infty(0,T_R))}^2
	\le I_5+I_6+I_7, \quad\mbox{where} \\
	& I_5 = \E\bigg(\sup_{0<t<T_R}\bigg|\sum_{i=1}^n\int_0^t\int_\dom
	\delta\na(u_i-v_i)(s)\cdot\na\psi_i \dd x\dd s\bigg|\bigg)^2, \nonumber \\
	& I_6 = \E\bigg(\sup_{0<t<T_R}\bigg|\sum_{i,j=1}^n\int_0^t\int_\dom
	a_{ij}\big(u_i^+(s)\na u_j(s)-v_i^+(s)\na v_j(s)\big)
	\cdot\na\psi_i\dd x\dd s\bigg|\bigg)^2, \nonumber \\
	& I_7 = \E\bigg(\sup_{0<t<T_R}\bigg|\sum_{i=1}^n\bigg(\int_0^t
	\big(\sigma_{ij}(u(s))-\sigma_{ij}(v(s))\big)\dd W_j(s),\psi_i\bigg)_{L^2(\dom)}
	\bigg|\bigg)^2. \nonumber 
\end{align}
The terms $I_5$ and $I_7$ are estimated as $I_2$ and $I_4$, respectively, giving
\begin{align*}
  I_5 \le C(N)\delta^2T^2\|u-v\|_{M_{T_R}}^2\|\psi\|_{L^2(\dom)}^2, \quad
	I_7 \le C(N)T\|u-v\|_{M_{T_R}}^2\|\psi\|_{L^2(\dom)}^2.
\end{align*}
Writing $u_i^+\na u_j-v_i^+\na v_j 
= u_i^+\na(u_j-v_j) + (u_i^+-v_i^+)\na v_j$
and using the Lipschitz continuity of $z\mapsto z^+$, 
the remaining term becomes
\begin{align*}
  I_6 
	&\le C_AT\E\bigg(\int_0^{T_R}\|u(s)\|_{L^2(\dom)}^2\|\na(u-v)(s)\|_{L^2(\dom)}^2
	\dd s\bigg)\|\na\psi\|_{L^\infty(\dom)}^2 \\
	&\phantom{xx}{}+ C_AT\E\bigg(\int_0^{T_R}\|(u-v)(s)\|_{L^2(\dom)}^2
	\|\na v(s)\|_{L^2(\dom)}^2\dd s\bigg)\|\na\psi\|_{L^\infty(\dom)}^2 \\
	&\le C_AC(N)R^2T^2\E\bigg(\sup_{0<s<T_R}\|(u-v)(s)\|_{L^2(\dom)}^2\bigg)
	\|\psi\|_{L^2(\dom)}^2 \\
	&= C(N,R)T^2\|u-v\|_{M_{T_R}}^2\|\psi\|_{L^2(\dom)}^2.
\end{align*}
We conclude from \eqref{2.aux2} that
$$
  \|S(u)-S(v)\|_{M_{T_R}}^2 \le C(N,R)(T^2+T)\|u-v\|_{M_{T_R}}^2.
$$
Then, choosing $T>0$ such that $C(N,R)(T^2+T)<1$, the mapping
$S:M_{T_R}\to M_{T_R}$ is a contraction. By Banach's fixed-point theorem,
there exists a unique fixed point $u^N\in M_{T_R}$ that solves
\eqref{2.eqN}--\eqref{2.icN} in $[0,T_R]$.
\end{proof}

\subsection{Uniform estimates}

We show that the solution $u^N$ to \eqref{2.eqN}--\eqref{2.icN} satisfies
suitable uniform estimates.
The following lemma is the key result of this subsection.

\begin{lemma}[Entropy inequality for $u^N$]\label{lem.eiN}
Let $T>0$ and let $u^N$ be the pathwise unique strong solution to 
\eqref{2.eqN}--\eqref{2.icN} on $[0,T\wedge\tau_R]$, 
constructed in Lemma \ref{lem.ex.uN}.
Then there exists a constant $C(\lambda,T)$, depending on $\lambda$ and $T$, such that
\begin{align*}
  \E&\bigg(\sup_{0<t<T}H(u^N(t))\bigg)
	+ 2\delta\lambda\E\bigg(\int_0^{T}\int_\dom
	|\na u^N(s)|^2\dd x\dd s\bigg) \\
	&\phantom{xx}{}+\sum_{i=1}^n\E\bigg(\int_0^{T}\int_\dom(u_i^N)^+(s)
	|\na p_i(u^N(s))|^2\dd x\dd s\bigg)
	\le \E H(u^N(0))C(\lambda,T),
\end{align*}
where $\lambda>0$ is the smallest eigenvalue of $PA$.
\end{lemma}

\begin{proof}
Since $PA$ is symmetric positive definite, there exists $(PA)^{1/2}$
and the process $Y(t)=(PA)^{1/2}u^N(t)$
is well defined. Hence, we can write the entropy \eqref{1.H} as 
$$
  H(u^N) = \frac12\sum_{i,j=1}^n\int_\dom\pi_i a_{ij}u_i^Nu_j^N\dd x
	= \frac12\int_\dom (u^N)^T(PA)u^N\dd x = \frac12\|(PA)^{1/2}u^N\|_{L^2(\dom)}^2,
$$
Let $t_R=t\wedge\tau_R$ and $T_R=T\wedge\tau_R$. 
We apply It\^o's lemma \cite[Theorem 4.2.5]{LiRo15} to $Y(t)$:
\begin{align}\label{2.HuN}
  & H(u^N(t_R)) = H(u^N(0)) + J_1 + \cdots + J_4, \quad\mbox{where} \\
	& J_1 = -\delta\sum_{i,j=1}^n\pi_i a_{ij}\int_0^{t_R}\int_\dom\na u_i^N(s)\cdot
	\na u_j^N(s)\dd x\dd s
	\le -\delta\lambda\int_0^{t_R}\int_\dom|\na u^N|^2\dd x\dd s, \nonumber \\
	& J_2 = -\sum_{i=1}^n\int_0^{t_R}\int_\dom (u_i^N)^+(s)
	|\na p_i(u^N(s))|^2\dd x\dd s \le 0, \nonumber \\
	& J_3 = \frac{1}{2}\int_0^{t_R}\int_\dom\operatorname{Tr}\big[\sigma(u^N(s))
	\mathrm{D}^2 h(u^N(s))\sigma(u^N(s))^*\big]\dd x\dd s, \nonumber \\
	& J_4 = \sum_{i,j=1}^n \pi_i a_{ij}\int_\dom\bigg(\int_0^{t_R} u_i^N(s)
	\sigma_{ij}(u^N(s))\dd W_j(s)\bigg)\dd x. \nonumber 
\end{align}
Using $\mathrm{D}^2 h(u^N)=PA$ and Hypothesis (H4),
\begin{align}\label{2.J3}
  J_3 &= \frac12\int_0^{t_R}\|(PA)^{1/2}\sigma(u^N(s))
	\|_{\mathcal{L}_2(U;L^2(\dom))}^2\dd s
	\le C\int_0^{t_R}\|\sigma(u^N(s))\|_{\mathcal{L}_2(U;L^2(\dom))}^2\dd s \\
	&\le CC_\sigma^2\int_0^{t_R}\|u^N(s)\|_{L^2(\dom)}^2\dd s
	\le C(\lambda)\int_0^{t_R} H(u^N(s))\dd s. \nonumber
\end{align}
Inserting this estimate into \eqref{2.HuN}, taking the supremum over $[0,T_R]$
and then the expectation, we find that
\begin{align}\label{2.EHuN}
  \E&\bigg(\sup_{0<t<T_R}H(u^N(t))\bigg)
	\le \E H(u^N(0)) - \delta\lambda\E\bigg(\sup_{0<t<T_R}\int_0^{t}\int_\dom
	|\na u^N(s)|^2\dd x\dd s\bigg) \\
	&\phantom{xx}{}-\E\bigg(\sup_{0<t<T_R}\sum_{i=1}^n\int_0^{t}\int_\dom 
	(u_i^N)^+(s)|\na p_i(u^N(s))|^2\dd x\dd s\bigg) \nonumber \\
	&\phantom{xx}{}+ C(\lambda)\E\bigg(\sup_{0<t<T_R}\int_0^t H(u^N(s))\dd s\bigg)
	+ \E\bigg(\sup_{0<t<T_R}J_4(t)\bigg). \nonumber
\end{align}
We deduce from the Burkholder--Davis--Gundy inequality
that
\begin{align*}
  \E\bigg(\sup_{0<t<T_R}J_4(t)\bigg) 
	&\le C\sum_{i,j,\ell=1}^n\E\bigg\{\sup_{0<t<T_R}\int_0^t\sum_{k=1}^\infty
	\bigg(\int_\dom u_j^N(s)\sigma_{i\ell}(u^N(s))\eta_k\dd x\bigg)^2\dd s\bigg\}^{1/2} \\
	&\le C\sum_{i,j,\ell=1}^n\E
	\bigg\{\sup_{0<t<T_R}\int_0^t\sum_{k=1}^\infty\|u_j^N(s)\|_{L^2(\dom)}^2
	\|\sigma_{i\ell}(u^N(s))\eta_k\|_{L^2(\dom)}^2\dd s\bigg\}^{1/2} \\
	&\le C\E\bigg(\int_0^{T_R}\|u^N(s)\|_{L^2(\dom)}^2
	\|\sigma(u^N(s))\|_{\mathcal{L}_2(U;L^2(\dom))}^2\dd s\bigg)^{1/2}.
\end{align*}
Therefore, using Hypothesis (H4) and Young's inequality,
\begin{align*}
  \E&\bigg(\sup_{0<t<T_R}J_4(t)\bigg)
	\le C_\sigma C\E\bigg\{\sup_{0<t<T_R}\|u^N(s)\|_{L^2(\dom)}
	\bigg(\int_0^{T_R}\|u^N(s)\|_{L^2(\dom)}^2\dd s\bigg)^{1/2}\bigg\} \\
	&\le C(\lambda)\E\bigg\{\sup_{0<t<T_R}\|(PA)^{1/2}u^N(s)\|_{L^2(\dom)}
	\bigg(\int_0^{T_R}\|(PA)^{1/2}u^N(s)\|_{L^2(\dom)}^2\dd s\bigg)^{1/2}\bigg\} \\
	&\le \frac12\E\bigg(\sup_{0<t<T_R}H(u^N(t))\bigg)
	+ C(\lambda)\int_0^{T_R}\E\bigg(\sup_{0<s<t}H(u^N(s))\bigg)\dd t.
\end{align*}
We insert this estimate into \eqref{2.EHuN},
\begin{align}\label{2.EHuN2}
  \frac12&\E\bigg(\sup_{0<t<T_R}H(u^N(t))\bigg)
	\le \E H(u^N(0)) - \delta\lambda\E\int_0^{T_R}\int_\dom
	|\na u^N(s)|^2\dd x\dd s \\
	&{}-\sum_{i=1}^n\E\int_0^{T_R}\int_\dom (u_i^N)^+(s)
	|\na p_i(u^N(s))|^2\dd x\dd s
	+ C(\lambda)\int_0^{T_R}\E\bigg(\sup_{0<s<t}H(u^N(s))\bigg)\dd t, \nonumber
\end{align}
and apply Gronwall's lemma to obtain
$$
  \E\bigg(\sup_{0<t<T_R}H(u^N(t))\bigg) \le \E H(u^N(0))e^{2C(\lambda)T}.
$$
Hence, the right-hand side of \eqref{2.EHuN2}
does not depend on the chosen sequence of stopping times $\tau_R$, 
and we can pass to the limit $R\to\infty$, finishing the proof.
\end{proof}


\begin{corollary}[Uniform estimates]\label{coro.est}
Let $T>0$ and let $u^N$ be the pathwise unique strong solution to 
\eqref{2.eqN}--\eqref{2.icN} on $[0,T\wedge\tau_R]$, 
constructed in Lemma \ref{lem.ex.uN}.
Then there exists a constant $C>0$, which is independent of $N$, such that
for all $i=1,\ldots,n$,
\begin{equation*}
  \E\bigg(\sup_{0<t<T}\|u_i^N(t)\|_{L^2(\dom)}^2\bigg)
	+ \delta\lambda\E\bigg(\int_0^T\|\na u_i^N(t)\|_{L^2(\dom)}^2\dd t\bigg) \le C.
\end{equation*}
\end{corollary}

We also need higher-order moment estimates.

\begin{lemma}[Higher-order moment estimates]\label{lem.higher}
Let $T>0$, $p>2$, and let $u^N$ be the pathwise unique strong solution to 
\eqref{2.eqN}--\eqref{2.icN} on $[0,T\wedge\tau_R]$, 
constructed in Lemma \ref{lem.ex.uN}.
Then there exists a constant $C>0$, which is independent of $N$, such that
$$
  \E\bigg(\sup_{0<t<T}\|u^N(t)\|_{L^2(\dom)}^p\bigg) \le C.
$$
\end{lemma}

\begin{proof}
Let $T_R=T\wedge\tau_R$. 
We start from \eqref{2.HuN}, neglect the terms $J_1$ and $J_2$, take the
supremum over $(0,T_R)$, raise the power $p/2$ on both sides, take the expectation,
and use H\"older's inequality:
\begin{align}\label{2.J56}
  & \E\bigg(\sup_{0<t<T_R}H(u^N(t))^{p/2}\bigg) \le C\E H(u^N(0))^{p/2}
	+ J_5 + J_6, \quad\mbox{where} \\
	& J_5 = CT^{p/2-1}\E\int_0^{T_R}\|(PA)^{1/2}\sigma(u^N(s))
	\|_{\mathcal{L}_2(U;L^2(\dom))}^p\dd s, \nonumber \\
	& J_6 = \E\bigg(\sup_{0<t<T_R}\bigg|\sum_{i,j,k=1}^n\pi_i a_{ij}\int_\dom\int_0^t 
	u_j^N(s)\sigma_{ik}(u^N(s))\dd W_k(s)\dd x\bigg|^{p/2}\bigg) \nonumber
\end{align}
The term $J_5$ can be estimated as in \eqref{2.J3}:
$$
  J_5 \le C(\lambda)\E\int_0^{T_R} H(u^N(s))^{p/2}\dd s
	\le C(\lambda)\int_0^{T_R}\E\bigg(\sup_{0<s<t}H(u^N(s))^{p/2}\bigg)\dd t.
$$
Similarly as in Lemma \ref{lem.eiN}, we apply the Burkholder--Davis--Gundy 
inequality and Young's inequality to find that
\begin{align*}
  J_6 &\le C\E\bigg(\int_0^{T_R}\|u^N(s)\|_{L^2(\dom)}^2\|\sigma(u^N(s))
	\|_{\mathcal{L}_2(U;L^2(\dom))}^2\dd s\bigg)^{p/4} \\
	&\le CT^{1-p/2}\E\bigg\{\bigg(\sup_{0<s<T_R}\|u^N(s)\|_{L^2(\dom)}^{p}\bigg)^{1/2}
	\bigg(\int_0^{T_R}\|u^N(s)\|_{L^2(\dom)}^p\dd s\bigg)^{1/2}\bigg\} \\
	&\le \frac12\E\bigg(\sup_{0<s<T_R}H(u^N(s))^{p/2}\bigg)
	+ C(T,\lambda)\int_0^{T_R}\E\bigg(\sup_{0<s<t}H(u^N(s))^{p/2}\bigg)\dd t.
\end{align*}
We insert these estimates into \eqref{2.J56} and apply Gronwall's lemma
to conclude that
$$
  \E\bigg(\sup_{0<t<T_R}H(u^N(t))^{p/2}\bigg) \le C(\lambda,T)\E H(u^N(0))^{p/2},
$$
and the positive definiteness of $(PA)^{1/2}$ finishes the proof.
\end{proof}

\subsection{Tightness}

The tightness of the sequence of laws of $u^N$ on a suitable subspace
is proved similarly as in \cite[Section 2.4]{DJZ19}.
We consider the following spaces:
\begin{itemize}
\item $C^0([0,T];H^3(\dom)')$ with the topology $\mathcal{T}_1$, induced by its
canonical norm. 
\item $L_w^2(0,T;H^1(\dom))$ is the space $L^2(0,T;H^1(\dom))$ with the weak topology
$\mathcal{T}_2$.
\item $L^2(0,T;L^2(\dom))$ with the topology $\mathcal{T}_3$ induced by its
canonical norm.
\item $C^0([0,T];L^2_w(\dom))$ is the space of weakly continuous functions
$v:[0,T]\to L^2(\dom)$ endowed with the weakest topology $\mathcal{T}_4$ such that
for all $h\in L^2(\dom)$, the mappings
$$
  C^0([0,T];L_w^2(\dom))\to C^0([0,T];\R), \quad u\mapsto (u(\cdot),h)_{L^2(\dom)},
$$
are continuous.
\end{itemize}

We define the following space:
\begin{equation}\label{3.ZT}
  Z_T := C^0([0,T];H^3(\dom)')\cap L_w^2(0,T;H^1(\dom))\cap L^2(0,T;L^2(\dom))
	\cap C^0([0,T];L_w^2(\dom)),
\end{equation}
endowed with the topology $\mathcal{T}$, which is the maximum of the topologies
$\mathcal{T}_1,\ldots,\mathcal{T}_4$ of the corresponding spaces, i.e.\ the
smallest topology containing $\cap_{i=1}^4\mathcal{T}_i$. 

\begin{lemma}[Tightness]
Let $u^N$ be the pathwise unique strong solution to 
\eqref{2.eqN}--\eqref{2.icN}, 
constructed in Lemma \ref{lem.ex.uN}. Then the set of laws of $(u^N)$ is tight 
on $(Z_T,\mathcal{T})$.
\end{lemma}

\begin{proof}
We apply \cite[Corollary 3.9]{BrMo14}
with $U=H^3(\dom)$, $V=H^1(\dom)$, and $H=L^2(\dom)$. Since $V\hookrightarrow H$
is compact, Corollary \ref{coro.est}
shows that conditions (a) and (b) of \cite[Corollary 3.9]{BrMo14} are satisfied.
It remains to verify that $(u^N)$ satisfies the Aldous condition in $H^3(\dom)'$.
For this, let $(\tau_N)_{N\in\N}$ be a sequence of $\mathbb{F}$-stopping times
such that $0\le\tau_N\le T$. Let $t\in[0,T]$, $i\in\{1,\ldots,n\}$, 
and $\phi\in H^3(\dom)$. We write \eqref{2.eqN} as
\begin{align}\label{2.K123}
  & \langle u_i^N(t),\phi\rangle = \langle K_1^N + K_2^N(t) + K_3^N(t),\phi\rangle, 
	\quad\mbox{where} \\
	& \langle K_1^N,\phi\rangle = (\Pi_N(u_i^0),\phi)_{L^2(\dom)}, \nonumber \\
	& \langle K_2^N(t),\phi\rangle = -\int_0^t\int_\dom\Pi_N\bigg(\delta\na u_i^N(s)
	+ \sum_{j=1}^n a_{ij}(u_i^N)^+(s)\na u_j^N(s)\bigg)
	\cdot\na\phi\dd x\dd s, \nonumber \\
	& \langle K_3^N(t),\phi\rangle 
	= \sum_{j=1}^n\bigg(\int_0^t\Pi_N\big(\sigma_{ij}(u^N(s))\big)
	\dd W_j(s),\phi\bigg)_{L^2(\dom)}, \nonumber
\end{align}
and $\langle\cdot,\cdot\rangle$ denotes the dual pairing between $H^3(\dom)'$
and $H^3(\dom)$. (If $d\ge 1$ is arbitrary, we replace $H^3(\dom)$ by $H^m(\dom)$, 
where $m\ge 1+d/2$, $m\in\N$.)
Let $\theta>0$. Then, using the embedding
$H^3(\dom)\hookrightarrow W^{1,\infty}(\dom)$ for $d\le 3$ and the
Cauchy--Schwarz inequality,
\begin{align*}
  \E|&\langle K_2^N(\tau_N+\theta)-K_2^N(\tau_N),\phi\rangle|
	\le \E\int_{\tau_N}^{\tau_N+\theta}\bigg(\delta\|\na u_i(t)\|_{L^2(\dom)}
	\|\na\phi\|_{L^2(\dom)} \\
	&\phantom{xx}{}+ \sum_{j=1}^n a_{ij}\|(u_i^N)^+(s)\|_{L^2(\dom)}
	\|\na u_j^N\|_{L^2(\dom)}\|\na\phi\|_{L^\infty(\dom)}\bigg)\dd t \nonumber \\
	&\le \delta\theta^{1/2}\bigg(\E\int_{\tau_N}^{\tau_N+\theta}
	\|\na u_i(t)\|_{L^2(\dom)}^2\dd t\bigg)^{1/2}\|\phi\|_{H^1(\dom)} \nonumber \\
	&\phantom{xx}{}+ C_A^{1/2}\theta^{1/2}
	\bigg\{\E\bigg(\sup_{0<t<T}\|u_i^N(t)\|_{L^2(\dom)}^2
	\bigg)\bigg\}^{1/2}\bigg\{\E\int_0^T\|\na u(t)\|_{L^2(\dom)}^2\dd t\bigg\}^{1/2}
	\|\phi\|_{H^3(\dom)}. \nonumber 
\end{align*}
Therefore, in view of the estimates of Corollary \ref{coro.est},
\begin{equation}\label{2.K2}
  \E|\langle K_2^N(\tau_N+\theta)-K_2^N(\tau_N),\phi\rangle|
	\le C\theta^{1/2}\|\phi\|_{H^3(\dom)}.
\end{equation}
For the stochastic term, we use the It\^o isometry and Hypothesis (H4):
\begin{align}\label{2.K3}
  \E|&\langle K_3^N(\tau_N+\theta)-K_3^N(\tau_N),\phi\rangle|^2 \\
	&\le \E\bigg(\int_{\tau_N}^{\tau_N+\theta}\|\sigma(u^N(t))
	\|_{\mathcal{L}_2(U;L^2(\dom))}^2\dd t\bigg)\|\phi\|_{L^2(\dom)}^2 \nonumber \\
	&\le C_\sigma\theta\E\bigg(\sup_{0<t<T}\|u^N(t)\|_{L^2(\dom)}^2\bigg)
	\|\phi\|_{L^2(\dom)}^2 \le C\theta\|\phi\|_{L^2(\dom)}^2. \nonumber 
\end{align}

Next, let $\zeta$, $\eps>0$. We conclude from the Chebyshev
inequality and \eqref{2.K2} that
\begin{align*}
  \Prob\big(&\|K_2^N(\tau_N+\theta)-K_2^N(\tau_N)\|_{H^3(\dom)'}\ge \zeta\big) \\
	&\le \frac{1}{\zeta}\E\|K_2^N(\tau_N+\theta)-K_2^N(\tau_N)\|_{H^3(\dom)'}
	\le \frac{C\theta^{1/2}}{\zeta}. 
\end{align*}
Choosing $\theta_1:=(\zeta\eps/C)^2$, we infer that
$$
  \sup_{N\in\N}\sup_{0<\theta<\theta_1}
	\Prob\big(\|K_2^N(\tau_N+\theta)-K_2^N(\tau_N)\|_{H^3(\dom)'}\ge \zeta\big)
	\le \eps.
$$
In a similar way, it follows from the Chebyshev inequality and \eqref{2.K3} that
$$
  \sup_{N\in\N}\sup_{0<\theta<\theta_2}
	\Prob\big(\|K_3^N(\tau_N+\theta)-K_3^N(\tau_N)\|_{L^2(\dom)}\ge \zeta\big)
	\le \eps.
$$
We infer from \eqref{2.K123} that
$$
  \sup_{N\in\N}\sup_{0<\theta<\min\{\theta_1,\theta_2\}}
	\Prob\big(\|u_i^N(\tau_N+\theta)-u_i^N(\tau_N)\|_{H^3(\dom)'}\ge \zeta\big)
	\le 2\eps.
$$
This shows that $(u_i^N)$ satisfies the Aldous condition in $H^3(\dom)'$.
\end{proof}

\subsection{Convergence of the approximate solutions}

According to \cite[Theorem 23]{DJZ19}, the topological space $(Z_T,\mathcal{T})$ 
possesses a sequence of continuous funtions $Z_T\to\R$ that separates points of $Z_T$.
We deduce from the theorem of Skorokhod--Jakubowski and the tightness of the laws 
of $(u^N)$ on $(Z_T,\mathcal{T})$ that there exists a subsequence of $(u^N)$,
which is not relabeled, a probability space $(\widetilde\Omega,\widetilde{\mathcal{F}},
\widetilde\Prob)$, and, on this space, $Z_T\times C^0([0,T];U_0)$-valued random
variables $(\widetilde u,\widetilde W)$, $(\widetilde{u}^N,\widetilde{W}^N)$
with $N\in\N$ such that $(\widetilde{u}^N,\widetilde{W}^N)$ has the same law as
$(u^N,W)$ on $\mathcal{B}(Z_T\times C^0([0,T];U_0)$ and
$$
  (\widetilde{u}^N,\widetilde{W}^N)\to (\widetilde w,\widetilde W)
	\quad\mbox{in }Z_T\times C^0([0,T];U_0)\ \Prob\mbox{-a.s. as }N\to\infty.
$$
We deduce from the definition of $Z_T$ that
\begin{align*}
  \widetilde{u}^N\to \widetilde u &\quad\mbox{in }C^0([0,T];H^3(\dom)')
	\mbox{ and }L^2(0,T;L^2(\dom)), \\
	\widetilde{u}^N\rightharpoonup \widetilde u &\quad\mbox{weakly in }
	L^2(0,T;H^1(\dom)), \\
	\widetilde{u}^N\to \widetilde u &\quad\mbox{in }C^0([0,T];L_w^2(\dom)), \\
	\widetilde{W}^N\to \widetilde W &\quad\mbox{in }C^0([0,T];U_0).
\end{align*}
The strong convergence of $(\widetilde u_i^N)$ implies that
$(\widetilde u_i^N)^+\to \widetilde u_i^+$ in $L^2(0,T;L^2(\dom))$
$\Prob$-a.s.
Using Corollary \ref{coro.est} and Lemma \ref{lem.higher} and arguing as in 
\cite[Section 2.5]{DJZ19}, we can prove that for $p>2$,
$$
  \widetilde{\E}\bigg(\int_0^T\|\widetilde u(t)\|_{H^1(\dom)}^2\dd t\bigg) \le C,
	\quad \widetilde{\E}\bigg(\sup_{0<t<T}\|\widetilde u(t)\|_{L^2(\dom)}^p\bigg)\le C.
$$

Let $\widetilde{\mathbb{F}}$ and $\widetilde{\mathbb{F}}^N$ be the filtrations
generated by $(\widetilde u,\widetilde W)$ and $(\widetilde{u}^N,\widetilde{W}^N)$,
respectively. The arguments of \cite[Prop.~B4]{BrOn13} allow us to show that these
new random variables are actually stochastic processes. The progressive
measurability of $\widetilde{u}^N$ is a consequence of \cite[Appendix B]{BHM13}.
Set $\widetilde{W}_j^{N,k}=(\widetilde{W}_j^N(t),\eta_k)_{L^2(\dom)}$. 
Then $\widetilde{W}_j^{N,k}(t)$ are independent, standard
$\widetilde{\F}_t$-Wiener processes. Passing to the limit $N\to\infty$
in the characteristic function, 
we infer that
$\widetilde{W}(t)$ are $\widetilde{F}_t$-martingales with the correct marginal
distributions. Then, by L\'evy's charaterization theorem, $\widetilde{W}(t)$ is
a cylindrical Wiener process. 
For the following lemma, we introduce
$$
  H_\nu^3(\dom) = \big\{v\in H^3(\dom):\na v\cdot\nu=0\mbox{ on }\pa\Omega\big\}.
$$

\begin{lemma}
Let $\phi\in L^2(\dom;\R^n)$ and $\psi\in H_\nu^3(\dom;\R^n)$. Then it holds for
$i=1,\ldots,n$ that
\begin{align*}
  & \lim_{N\to\infty}\widetilde{\E}\int_0^T\big(\widetilde{u}^N(t)-\widetilde u(t),
	\phi\big)_{L^2(\dom)}\dd t = 0, \quad
	\lim_{N\to\infty}\widetilde{\E}\big(\widetilde{u}^N(0)-\widetilde u(0),\phi
	\big)_{L^2(\dom)}^2 = 0, \\
	& \lim_{N\to\infty}\widetilde{\E}\int_0^T\bigg|\int_0^t\int_\dom
	\bigg(\delta\na\widetilde{u}^N(s)
	+ \sum_{j=1}^n a_{ij}(\widetilde{u}_i^N)^+(s)\na\widetilde{u}_j^N(s) \\
	&\phantom{xx}{}- \delta\na\widetilde u(s)
	- \sum_{j=1}^n a_{ij}\widetilde u_i^+(s)\na\widetilde u_j(s)\bigg)
	\cdot\na\psi_i\dd x\dd s\bigg|\dd t = 0, \\
	& \lim_{N\to\infty}\widetilde{\E}\int_0^T\bigg|\sum_{j=1}^n\int_0^t
	\big(\sigma_{ij}(\widetilde{u}^N(s))\dd \widetilde{W}_j^N(s)
	- \sigma_{ij}(\widetilde{u}(s))\dd\widetilde W_j(s),\phi_i\big)_{L^2(\dom)}
	\bigg|^2\dd t = 0.
\end{align*}
\end{lemma}

\begin{proof}
The proof of this lemma is very similar to that one for \cite[Lemma 16]{DJZ19},
using the previous convergence results and Vitali's convergence theorem. 
The only difference is the convergence in the nonlinear terms
$(\widetilde{u}_i^N)^+\na\widetilde{u}_j^N$. In fact, the strong
convergence of $(\widetilde{u}_i^N)^+$ in $L^2(0,T;L^2(\dom))$
and the weak convergence of $\na \widetilde{u}_j^N$ in $L^2(0,T;H^1(\dom))$ 
imply that
\begin{align*}
  \bigg|\int_0^t&\int_\dom\big((\widetilde{u}_i^N)^+(s)\na\widetilde{u}_j^N(s)
	- \widetilde u_i^+(s)\na\widetilde u_j(s)\big)\cdot\na\psi_i\dd x\dd s\bigg| \\
	&\le \|(\widetilde{u}_i^N)^+-\widetilde{u}_i^+\|_{L^2(0,t;L^2(\dom)}
	\|\widetilde{u}_j^N\|_{L^2(0,t;H^1(\dom))}\|\psi_i\|_{W^{1,\infty}(\dom)} \\
	&\phantom{xx}{}+ \bigg|\int_0^t\int_\dom(\widetilde{u}_i^N)^+(s)
	\na(\widetilde{u}_j^N-\widetilde{u}_i)(s)\cdot\na\psi_i\dd x\dd s\bigg|\to 0
	\quad\Prob\mbox{-a.s. as }N\to\infty.
\end{align*}
Since $(\widetilde{u}_i^N)^+\na\widetilde{u}_j^N$ is uniformly bounded,
Vitali's convergence theorem gives
$$
  \lim_{N\to\infty}\widetilde{\E}\int_0^T\bigg|\int_0^t\int_\dom\big(
	(\widetilde{u}_i^N)^+(s)\na\widetilde{u}_j^N(s)
	-\widetilde{u}_i^+(s)\na\widetilde{u}_j(s)
	\big)\cdot\na\psi_i\dd x\dd s\bigg|\dd t = 0.
$$
This finishes the proof.
\end{proof}

We consider the maps
$L_i^N$, $L_i:Z_T\times C^0([0,T];U_0)\times H_\nu^3(\dom)\to 
L^1(\widetilde\Omega\times(0,T))$ for $i=1,\ldots,n$, defined by
\begin{align*}
  L_i^N(\widetilde{u}^N,\widetilde{W}^N,\psi)(t)
	&= (\Pi_N(\widetilde{u}_i(0)),\psi)_{L^2(\dom)} \\
	&\phantom{xx}{}- \int_0^t\int_\dom\Pi_N\bigg(\delta\na\widetilde{u}_i^N(s)
	+ \sum_{j=1}^n a_{ij}(\widetilde{u}_i^N)^+(s)\na\widetilde{u}_j^N(s)\bigg)
	\cdot\na\psi_i\dd x\dd s \\
	&\phantom{xx}{}+ \sum_{j=1}^n\bigg(\int_0^t\Pi_N\sigma_{ij}(\widetilde{u}^N(s))
	\dd\widetilde{W}_j^N(s),\psi_i\bigg)_{L^2(\dom)}, \\
	L_i(\widetilde u,\widetilde W,\psi)(t)
	&= (\widetilde{u}_i(0),\psi_i)_{L^2(\dom)} \\
	&\phantom{xx}{}- \int_0^t\int_\dom\bigg(\delta\na\widetilde{u}_i(s)
	+ \sum_{j=1}^n a_{ij}\widetilde{u}_i^+(s)\na\widetilde{u}_j(s)\bigg)
	\cdot\na\psi_i\dd x\dd s \\
	&\phantom{xx}{}+ \sum_{j=1}^n\bigg(\int_0^t\sigma_{ij}(\widetilde{u}(s))
	\dd\widetilde{W}_j(s),\psi_i\bigg)_{L^2(\dom)}
\end{align*}
The previous lemma implies the following result.

\begin{corollary}\label{coro.lim}
It holds for any $\phi\in L^2(\dom;\R^n)$ and $\psi\in H^3_\nu(\dom;\R^n)$ that
\begin{align*}
  \lim_{N\to\infty}\|(\widetilde{u}^N,\phi)_{L^2(\dom)}
	- (\widetilde u,\phi)_{L^2(\dom)}\|_{L^2(\widetilde\Omega\times(0,T))} &= 0, \\
	\lim_{N\to\infty}\|L_i^N(\widetilde{u}^N,\widetilde{W}^N,\psi)
	- L_i(\widetilde{u},\widetilde W,\psi)\|_{L^1(\widetilde\Omega\times(0,T))} &= 0.
\end{align*}
\end{corollary}

\subsection{End of the proof}\label{sec.end}

Since $u^N$ is a strong solution to \eqref{2.eqN}--\eqref{2.icN}, we have
$$
  (u_i^N(t),\psi)_{L^2(\dom)} = L_i^N(u^N,W,\psi)(t) \quad\Prob\mbox{-a.s.}
$$
for all $t\in[0,T]$, $\psi\in H^1(\dom)$, and $i=1,\ldots,n$. In particular,
$$
  \E\int_0^T\big|(u_i^N(t),\psi)_{L^2(\dom)} - L_i^N(u^N,W,\psi)(t)\big|\dd t = 0.
$$
As the laws of $(u^N,W)$ and $(\widetilde{u}^N,\widetilde{W}^N)$ are the same,
$$
  \widetilde{\E}\int_0^T\big|(\widetilde{u}_i^N(t),\psi)_{L^2(\dom)} 
	- L_i^N(\widetilde{u}^N,\widetilde{W}^N,\psi)(t)\big|\dd t = 0.
$$
By Corollary \ref{coro.lim}, we can pass to the limit $N\to\infty$, yielding
$$
  \widetilde{\E}\int_0^T\big|(\widetilde{u}_i(t),\psi)_{L^2(\dom)} 
	- L_i(\widetilde{u},\widetilde{W},\psi)(t)\big|\dd t = 0
$$
for all $\psi\in H_\nu^3(\dom)$. 
Since the space $H_\nu^3(\dom)$ is dense in $H^1(\dom)$,
this identity also holds for all $\psi\in H^1(\dom)$. This means that
$(\widetilde{u}_i(t),\psi)_{L^2(\dom)} = L_i(\widetilde{u},\widetilde{W},\psi)(t)$
$\Prob$-a.s.\ for a.e.\ $t\in(0,T)$, for all $\psi\in H^1(\dom)$ and $i=1,\ldots,n$.
Thus, setting $\widetilde V=(\widetilde\Omega,\widetilde\F,\widetilde{\mathbb{F}},
\widetilde{\Prob})$, we conclude that $(\widetilde V,\widetilde u,\widetilde W)$
is a martingale solution to \eqref{1.eq}--\eqref{1.bic}. 

It remains to verify that $\widetilde u$ is nonnegative componentwise in
$\dom\times(0,T)$ $\Prob$-a.s. This follows from the stochastic Stampacchia
truncation argument developed in \cite{CPT16}. The idea is to approximate
the test function $g(z)=z^-:=\max\{0,-z\}$ for $z\in\R$ by a smooth function
$g_\eps$ with linear growth and to apply It\^o's formula
to $G_\eps(v)=\int_\dom g_\eps(v(x))^2\dd x$.
It is proved in \cite[Section 2.6]{DJZ19} 
that the limit $N\to\infty$ and then $\eps\to 0$ leads to
\begin{align*}
  \E\|\widetilde u_i^-(t)\|_{L^2(\dom)}^2 &\le \E\|(u_i^0)^-\|_{L^2(\dom)}^2
	+ \sum_{j=1}^n\E\int_0^t\|\sigma_{ij}(-\widetilde u_i^-(s))
	\|_{\mathcal{L}_2(U;L^2(\dom))}^2\dd s \\
	&\le C_\sigma\sum_{j=1}^n\E\int_0^t
	\|\widetilde u_i^-(s))\|_{L^2(\dom)}^2\dd s,
\end{align*}
where we used that $\widetilde u_i^+\na\widetilde u_i^-=0$.
It follows from Gronwall's lemma that 
$\E\|\widetilde u_i^-(t)\|_{L^2(\dom)}^2=0$ for
$t\in(0,T)$ and consequently $\widetilde u_i(t)\ge 0$ 
in $\dom$ $\Prob$-a.s.\ for a.e.\ $t\in(0,T)$ and all $i=1,\ldots,n$. 


\section{Proof of Theorem \ref{thm.time}}\label{sec.time}

We apply It\^o's lemma \cite[Theorem 4.2.5]{LiRo15} 
to the process $t\mapsto e^{\eta t/2}(PA)^{1/2}(u^N-\bar{u}^N)(t)$, where $\eta>0$
will be determined later and we set
$\bar{u}^N(t):=|\dom|^{-1}\int_\dom u^N(x,t)\dd x$.
Recall that $H(u^N|\bar{u}^N)=\frac12\|(PA)^{-1/2}(u^N-\bar{u}^N)\|_{L^2(\dom)}^2$. 
Then $\bar u^N$ solves
$$
  \dd\bar{u}_i^N(t) = \sum_{j=1}^n\bar{\sigma}_{ij}(u^N(t))\dd W_j(t), \quad
	\bar{\sigma}_{ij}(u^N) = \frac{1}{|\dom|}
	\int_\dom\sigma(u^N)\dd x.
$$
We compute
\begin{align}\label{3.Hrel}
  e^{\eta t}&H(u^N(t)|\bar{u}^N(t)) - H(u^N(0)|\bar{u}^N(0)) - \eta\int_0^t e^{\eta s}
	H(u(s)|\bar{u}^N(s))\dd s \\
	&= -\delta\sum_{i,j=1}^n\pi_i a_{ij}\int_0^t e^{\eta s}
	\int_\dom\na u_i^N(s)\cdot\na u_j^N(s)\dd x\dd s \nonumber \\
	&\phantom{xx}{}- \sum_{i=1}^n\pi_i\int_0^t e^{\eta s}\int_\dom 
	(u_i^N)^+(s)|\na p_i(u^N(s))|^2\dd x\dd s \nonumber \\
	&\phantom{xx}{}+ \frac12\int_0^t e^{\eta s}\int_\dom\operatorname{Tr}
	\big[(\sigma(u^N)-\bar{\sigma}(u^N))\mathrm{D}^2_u H(u^N|\bar{u}^N)
	(\sigma(u^N)-\bar{\sigma}(u^N))
	 \big](s)\dd x\dd s \nonumber \\
	&\phantom{xx}{}+ \sum_{i,j,k=1}^n\pi_i a_{ij}\int_\dom\int_0^t e^{\eta s}
	(u_j^N-\bar{u}^N_j)(s)\big(\sigma_{ik}(u^N(s))-\bar{\sigma}_{ik}(u^N(s))\big)
	\dd W_k(s) \dd x. \nonumber
\end{align}
With $\lambda>0$ being the smallest eigenvalue of the positive definite matrix
$(\pi_i a_{ij})$, the first term on the right-hand side is estimated according to
$$
  -\delta\sum_{i,j=1}^n\pi_i a_{ij}\int_0^t e^{\eta s}
	\int_\dom\na u_i^N(s)\cdot\na u_j^N(s)\dd s\dd s 
	\ge -\delta\lambda\int_0^t e^{\eta s}\int_\dom|\na u^N(s)|^2\dd x\dd s.
$$
Then, taking the expecation in \eqref{3.Hrel}, neglecting the second term
on the right-hand side of \eqref{3.Hrel}, and observing that the stochastic
integral vanishes, we obtain
\begin{align*}
  e^{\eta t}&\E H(u(t)|\bar{u}(t)) - \E H(u(0)|\bar{u}(0)) - \eta\int_0^t e^{\eta s}
	\E H(u(s)|\bar{u}(s))\dd s \\
	&\le -\delta\lambda\int_0^t e^{\eta s}\int_\dom|\na u(s)|^2\dd x\dd s
	+ C\E\int_0^t e^{\eta s}\|(PA)^{1/2}(\sigma(u)-\bar{\sigma}(u))
	\|_{\mathcal{L}_2(U;L^2(\dom))}^2\dd s.
\end{align*}
We estimate the last term on the right-hand side:
\begin{align*}
  \|(P&A)^{1/2}(\sigma(u)-\bar{\sigma}(u))
	\|_{\mathcal{L}_2(U;L^2(\dom))}^2
	= \sum_{k=1}^\infty\big\|(PA)^{1/2}
	\big(\sigma(u)-\bar{\sigma}(u)
	\big)\eta_k\big\|_{L^2(\dom)}^2 \\
	&\le C\sum_{i,j=1}^n\sum_{k=1}^\infty
	\|\sigma_{ij}(u(s))\eta_k-\bar{\sigma}_{ij}(u(s))\eta_k\|_{L^2(\dom)}^2 \\
	&= C\sum_{i,j=1}^n\sum_{k=1}^\infty\int_\dom
	\bigg(\sigma_{ij}(u(s))\eta_k-\frac{1}{|\dom|}
	\int_\dom\sigma_{ij}(u(s))\eta_k\dd y\bigg)^2\dd x \\
	&= \frac{C}{|\dom|^2}\sum_{i,j=1}^n\sum_{k=1}^\infty
	\int_\dom\bigg(
	\int_\dom\big(\sigma_{ij}(u(x,s))\eta_k-\sigma_{ij}(u(s,y))\eta_k\big)
	\dd y\bigg)^2\dd x \\
	&\le \frac{C}{|\dom|}\sum_{i,j=1}^n\sum_{k=1}^\infty
	\int_\dom\int_\dom\big(\sigma_{ij}(u(x,s))\eta_k-\sigma_{ij}(u(s,y))\eta_k\big)^2
	\dd y\dd x,
\end{align*}
where we used the Cauchy--Schwarz inequality in the last step. 
Next, we add and subtract the space-independent term $\sigma_{ij}(\bar{u}(s))\eta_k$
and apply the triangle inequality:
\begin{align*}
  \|(P&A)^{1/2}(\sigma(u)-\bar{\sigma}(\bar{u}))
	\|_{\mathcal{L}_2(U;L^2(\dom))}^2 \\
	&\le \frac{C}{|\dom|}\sum_{i,j=1}^n\sum_{k=1}^\infty
	\int_\dom\int_\dom\big(\sigma_{ij}(u(x,s))\eta_k-\sigma_{ij}(\bar{u}(s))\eta_k
	\big)^2\dd y\dd x \\
	&\phantom{xx}{}+ \frac{C}{|\dom|}\sum_{i,j=1}^n\sum_{k=1}^\infty
	\int_\dom\int_\dom\big(\sigma_{ij}(u(y,s))\eta_k-\sigma_{ij}(\bar{u}(s))\eta_k
	\big)^2\dd y\dd x \\
	&= 2C\sum_{i,j=1}^n	\|\sigma_{ij}(u(s))-\sigma_{ij}(\bar{u}(s))
	\|_{\mathcal{L}_2(U;L^2(\dom))}^2 
	\le 2C_1C_\sigma^2\|u(s)-\bar{u}(s)\|_{L^2(\dom)}^2.
\end{align*}
Thus, by the Poincar\'e--Wirtinger inequality with constant $C_P>0$,
\begin{align*}
  e^{\eta t}& \E H(u(t)|\bar{u}(t)) - \E H(u(0)|\bar{u}(0)) \\
	&\le \bigg(C(\lambda)\eta + 2C_1C_\sigma^2 - \frac{\delta\lambda}{C_P^2}\bigg)
	\int_0^t e^{\eta s}\int_\dom\|(u-\bar{u})(s)\|_{L^2(\dom)}^2\dd s. 
\end{align*}
Without loss of generality, we may choose $C_1\ge 1/2$. We set
$c_0^2:=\delta\lambda/(2C_1C_P^2)$, choose $0<C_\sigma<c_0$, and set
$\eta:=(c_0^2-C_\sigma^2)/C(\lambda)$. Then
$C(\lambda)\eta + 2C_1C_\sigma^2 - \delta\lambda/C_P^2 
= (1-2C_1)(c_0^2-C_\sigma^2) < 0$,
and we end the proof.


\section{Numerical illustration}

We discretize \eqref{1.eq} for $n\ge 2$ species
by a semi-implicit Euler--Maruyama scheme and 
centered finite differences in one space dimension. 
Let $\dom=(0,1)$, $J\in\N$, $\Delta x=1/J$, $\Delta t>0$,
$x_j=j\Delta x$ for $j=0,\ldots,J$, and
$t_k=k\Delta t$ for $k\in\N_0$. We approximate $u_i(x_j,t_k,\cdot)$ by
$u_{ij}^k$, solving
$$
  u_{ij}^{k+1} = (I+\delta\Delta t A_\Delta)^{-1}\bigg(u_{ij}^k 
	+ \frac{\Delta t}{\Delta x}(F_{i,j+1/2}^k-F_{i,j-1/2}^k)
	+ \sum_{j=1}^n\sigma_{ij}(u^k)\Delta W_j^k\bigg),
$$
where the centered discrete fluxes are given by
$$
  F_{i,j+1/2}^k = \frac{1}{2\Delta x}(u_{i,j+1}^k+u_{ij}^k)\sum_{\ell=1}^n a_{i\ell}
	(u_{\ell,j+1}^k-u_{\ell,j}^k),
$$
and we have set $u^k=(u_{ij}^k)_{i=1,2,\,j=0,\ldots,J}$. 
The matrix $A_\Delta\in\R^{(J+1)\times(J+1)}$ is associated to 
the discrete Laplacian
with Neumann conditions, and $\Delta W_j^k = W_j(t_{k+1})-W_j(t_k)$.

First, we choose the initial data
$$
  u_1^0(x) = \mathrm{1}_{[0,1/2]}(x), \quad 
	u_2^0(x) = 10x^2\bigg(\frac12-\frac{x}{3}\bigg)\quad\mbox{for }x\in[0,1],
$$
the diffusion parameter $\delta=1$,
the coefficients $a_{11}=a_{22}=2$, $a_{12}=a_{21}=1$, and the
stochastic diffusion $\sigma_{ij}(u)=0.001\sqrt{1+u_i}\delta_{ij}$ for $i,j=1,2$.
Observe that the matrix $(a_{ij})$ is symmetric positive definite.
The $\ell^2$ errors of $(u_1,u_2)$ 
versus time step size $\Delta t$ (with $\Delta x=0.02$ and $2^{11}$ samples)
and space grid size $\Delta x$ 
(with $\Delta t=10^{-5}$ and $2^{11}$ samples) at time $t=1$
are presented in Figure \ref{fig.error}. For the reference solution, we have
chosen $\Delta t=5\cdot 10^{-5}$, $\Delta x=0.02$ (left figure) and
$\Delta t=10^{-5}$ and $\Delta x=1/128$ (right figure).
As expected, we observe a
half-order convergence in time and first-order convergence in space.

\begin{figure}[ht]
\includegraphics[width=0.48\textwidth]{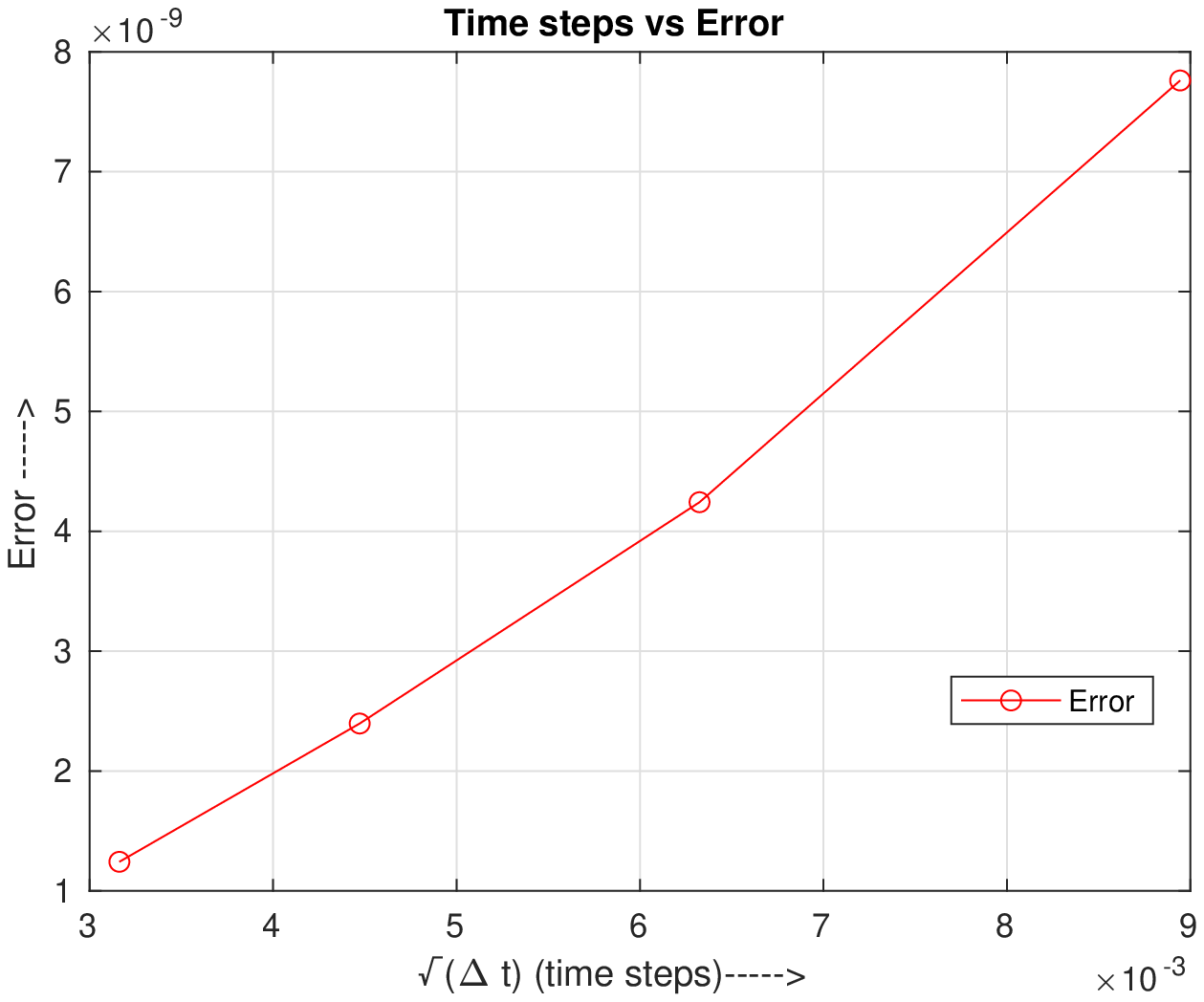}
\includegraphics[width=0.48\textwidth]{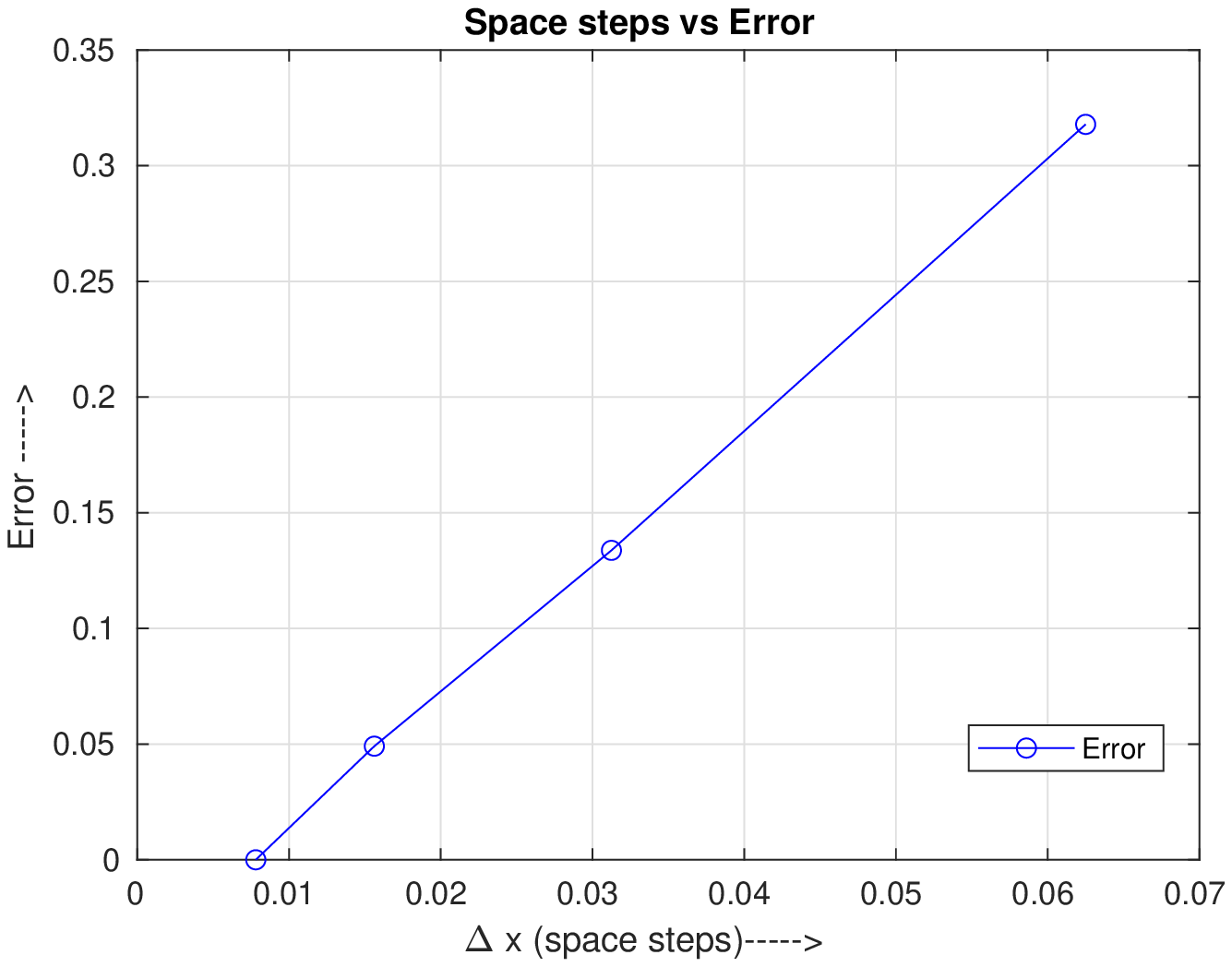}
\caption{Discrete $\ell^2$ error of $(u_1,u_2)$ versus time step size (left) and space
grid size (right). Observe that the $x$-axis of the left figure is scaled
with $\sqrt{\Delta t}$.}
\label{fig.error}
\end{figure}

Next, we present the long-time behavior of the discrete $\ell^2$ norms of
the densities in the three-species case. 
The parameters are $\delta=1$, $a_{12}=a_{23}=a_{31}=1$ and
$a_{ij}=0$ else, $\sigma_{ij}=0.1\sqrt{1+u_i}\delta_{ij}$ for $i,j=1,2,3$.
Notice that the matrix $(a_{ij})$ does not satisfy the detailed-balance
condition such that, strictly speaking, Theorem \ref{thm.time} is not applicable.
The time step is $\Delta t=10^{-4}$ and the grid size $\Delta x=0.02$.
Figure \ref{fig.n3} shows the time evolution of the $\ell^2$ norms of
$u_i$ in the deterministic case, while the right figure illustrates
the dynamics of the $\ell^2$ norms in the stochastic situation (with $2^{11}$ samples).
We observe that the $\ell^2$ norms, which are equivalent
to the Rao entropy, converge (in expectation) as $t\to\infty$,
and the behavior is similar in both cases.

\begin{figure}[ht]
\includegraphics[width=0.48\textwidth]{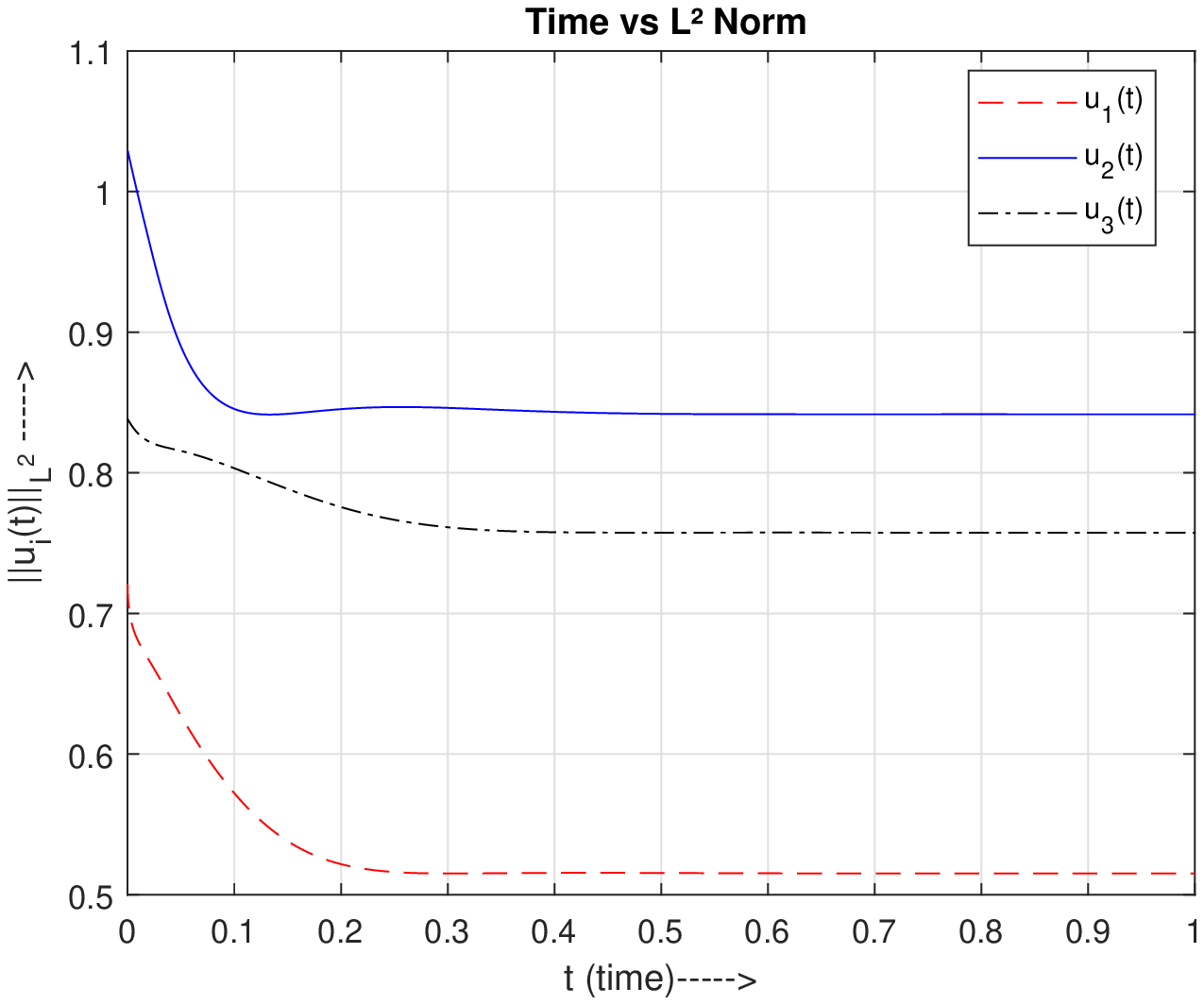}
\includegraphics[width=0.48\textwidth]{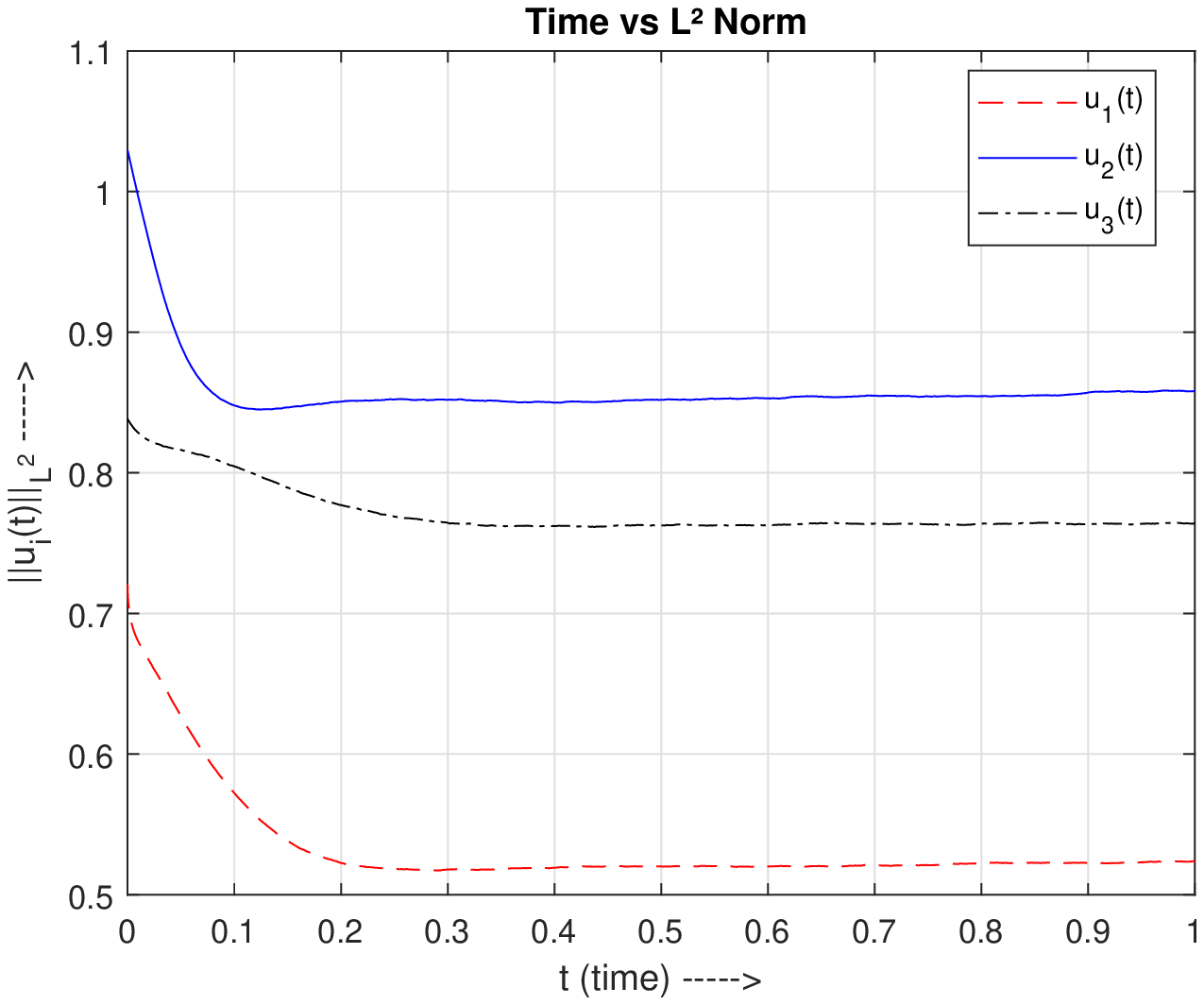}
\caption{Long-time behavior of the $\ell^2$ norms of $u_i$ in the deterministic
(left) and stochastic (right) case.}
\label{fig.n3}
\end{figure}


\end{document}